\documentclass[12pt]{article}

\usepackage{amsmath}
\usepackage{amsthm}
\usepackage{amsfonts}
\usepackage{amssymb}
\usepackage{graphicx}
\usepackage{latexsym}
\usepackage{mathrsfs}
\usepackage{epsfig}

\advance \topmargin by -\headheight
\advance \topmargin by -\headsep
\evensidemargin \oddsidemargin
\newtheorem{theorem}{Theorem}[section]
\newtheorem{thm}{Theorem}[section]

\newtheorem{prop}[theorem]{Proposition}

\theoremstyle{definition}

\newtheorem{defn}[theorem]{Definition}

\newtheorem{cor}[theorem]{Corollary}

\theoremstyle{remark}
\newtheorem{remark}[theorem]{Remark}

\numberwithin{equation}{section}

\newcommand\EE{\mathbb E}

\newcommand\NN{\mathbb N}
\newcommand\RR{\mathbb R}
\newcommand\R{\mathbb R}
\newcommand\ZZ{\mathbb Z}
\newcommand\Z{\mathbb Z}

\newcommand\FF{\mathbb F}

\newcommand\cB{\mathfrak{B}}

\newcommand\cI{\mathcal{I}}

\newcommand\cS{\mathcal{S}}

\newcommand\cO{\mathcal{O}}
\newcommand\cR{\mathcal{R}}

\newcommand\abs[1]{\left|#1\right|}

\newcommand\inn[1]{\left\langle #1 \right\rangle}
\newcommand\set[1]{\left\{{#1}\right\}}

\def\cc{\curvearrowright}

\def\fh{{\mathfrak{v}}}

\def\sA{{\mathbb{A}}}

\def\sM{{\mathbb{M}}}
\def\MM{{\mathbb{M}}}
\def\M{{\mathbb{M}}}

\def\II{{\mathbb{I}}}

\def\A{{\mathbb{A}}}

\def\Aut{{\textrm{Aut}}}
\def\bnu{\bar{\nu}}
\def\RS{{\textrm{RS}}}
\def\SRS{{\textrm{SRS}}}

\begin{document}
\title{Amenable equivalence relations and the construction of ergodic averages for group actions}

\author{Lewis Bowen\footnote{supported in part by NSF grant DMS-0968762, NSF CAREER Award DMS-0954606 and BSF grant 2008274}  ~and Amos Nevo\footnote{supported in part ISF grant  and BSF grant 2008274}}
%






\maketitle

\begin{abstract}
We present a general new method for constructing pointwise ergodic sequences on countable groups, which is applicable to amenable as well as to non-amenable groups and treats both cases on an equal footing.  The principle underlying the method is that  both cases can be viewed as instances of the general ergodic theory of amenable equivalence relations. 
 \end{abstract}


\section{Introduction}

\subsection{Pointwise ergodic families}

Let $\Gamma$ be a countable group with a probability-measure-preserving (pmp)  action on a standard probability space $(X,\mu)$. Any probability measure $\zeta$ on $\Gamma$ determines an operator on $L^1(X,\mu)$ defined by
$$\pi_X(\zeta)(f):= \sum_{\gamma\in \Gamma} \zeta(\gamma) f\circ \gamma^{-1},\quad \forall f\in L^1(X,\mu).$$

\begin{defn}
Let $\II$ denote either $\RR_{>0}$ or $\NN$. Suppose  $\{\zeta_r\}_{r\in \II}$ is a family of probability measures on $\Gamma$. 
\begin{enumerate}\item  If for every pmp action $\Gamma \cc (X,\mu)$ and every $f\in L^p(X,\mu)$ the functions $ \pi_X(\zeta_r)(f)$ converge pointwise a.e. as $ r\to \infty$   then we say $\{\zeta_r\}_{r\in \II}$ is a {\em pointwise convergent family in $L^p$}. If, in addition, the a.e.-pointwise limit of $\pi_X(\zeta_r)(f)$ is the conditional expectation of $f$ on the $\sigma$-algebra of $\Gamma$-invariant Borel sets then $\{\zeta_r\}_{r\in \II}$ is a {\em pointwise ergodic family in $L^p$}. 
\item We say that $\{\zeta_r\}_{r\in \II}$ {\em satisfies the strong $L^p$ maximal inequality} if there is a constant $C_p>0$ depending only on $\{\zeta_r\}_{r\in \II}$ such that for every $f \in L^p(X,\mu)$, $\|\sM[f|\zeta]\|_p \le C_p \|f\|_p$ where $\sM[f|\zeta] = \sup_{r \in \II} \pi_X(\zeta_r)(|f|)$. 
\item Similarly,  $\{\zeta_r\}_{r\in\II}$ is said to {\em satisfy the weak $(1,1)$-type maximal inequality} if there is a constant $C_{1,1}>0$ depending only on $\{\zeta_r\}_{r\in \II}$ such that for every $f \in L^1(X,\mu)$ and $t>0$,
$$\mu(\{ x\in X:~ \M[f|\zeta](x)\ge t\}) \le \frac{ C_{1,1} \|f\|_1}{t}.$$
\end{enumerate}
It follows from standard interpolation arguments that if $\{\zeta_r\}_{r\in \II}$ satisfies the weak (1,1)-type maximal inequality then it also satisfies the strong $L^p$ maximal inequality for all $p>1$. Moreover in that case it satisfies the {\em $L\log(L)$-maximal inequality} which means there is a constant $C_1>0$ depending only on $\{\zeta_r\}_{r\in \II}$ such that for every $f \in L\log(L)(X,\mu)$, $\|\sM[f|\zeta]\|_{1} \le C_1 \|f\|_{L\log L}$ where, for any function $\phi$ of $X$, $\| \phi \|_{L\log L} := \int_X |\phi| \log^+ |\phi |~d\mu$. 
\end{defn}

 Traditionally, since the time of von-Neumann and Birkhoff, much of the effort in ergodic theory has been devoted to the study of ergodic actions of Abelian and more generally amenable groups (see  \cite{OW87}), with the averages studied most often being uniform averages on an asymptotically invariant sequence of sets (a F\o lner sequence).  For more on the subject of  ergodic theorems on amenable groups we refer to \cite{We03},  the survey \cite{Ne05}, and for a more detailed exposition to \cite{AAB}. We also refer to \cite{Av} for an ergodic theorem pertaining to cross sections of actions of amenable groups.

Our main purpose in the present paper is to give a general new  method for constructing probability measures $\zeta_r$ on $\Gamma$ which form pointwise convergent and pointwise ergodic families on general countable groups, which may be amenable or non-amenable.  The basic ingredient underlying our approach is the realization that one can utilize the amenable actions of a group  in order to construct families of ergodic averages on it. The generality of this approach is underscored by the fact that for any countable group $\Gamma$, the Poisson boundary associated with a generating probability measure is an amenable ergodic action of $\Gamma$ \cite{Z2}, 
 so that any countable group admits such an action.


 Some instances of our approach were developed in \cite{BN1} and in \cite{BN2}, and in the present paper, we greatly generalize the constructions used there. We show that  pointwise ergodic families on $\Gamma$ can be obtained using  $\Gamma$-valued cocycles defined  on a pmp amenable  equivalence relation, subject to certain natural necessary conditions. 
 This construction gives rise to a variety of ergodic families on countable groups, some of which will be discussed briefly in \S \ref{sec:example}.  To further demonstrate the significance and utility of our approach we note that  in \cite{BN4}, the main results of the present paper are crucial  to the proof establishing pointwise convergence for geometric averages in Gromov hyperbolic groups, supported in spherical shells. In \cite{BN5}, the methods and results of the present paper are applied to prove pointwise ergodic theorems for Markov groups. More details on these applications will be given in \S 6 below, when the necessary notation and definitions are available. 

\subsection{From amenable groups to amenable equivalence 
relations}
 
We proceed to outline our general method for constructing pointwise ergodic sequences for pmp actions of countable groups. Along the way, we will briefly state the main results of the present paper and place them in their natural context.  




\subsubsection{Pointwise ergodic theorems for measured equivalence relations} The concept of pointwise ergodic sequence can be generalized from groups to measured equivalence relations. This generalization is very useful,  since it is easy to obtain such sequences for amenable measured equivalence relations,  and it is then possible to push them forward to groups via cocycles. Before explaining this connection, we explain here how to generalize ergodic theorems to measured equivalence relations. Let then $(B,\nu)$ be a standard probability space and $\cR \subset B \times B$ be a discrete measurable equivalence relation. We assume $\nu$ is $\cR$-invariant: if $c$ denotes counting measure on $B$ then $c \times \nu|_\cR = \nu \times c|_\cR$. 

Let  $\Omega=\{\omega_i\}_{i\in \II}$ be a measurable leafwise collection of probability measures $\omega_i:\cR \to [0,1]$. More precisely, we require that the map $(i, (b,c)) \in \II \times \cR \to \omega_i(b,c) \in [0,1]$ is measurable and for a.e. $b\in B$, $\sum_{c \in B} \omega_i(c,b)=1$ where, for convenience, we set $\omega_i(c,b)=0$ if $(c,b) \notin \cR$. 

Let $(X,\mu)$ be another standard probability space and $\alpha:\cR \to \Aut(X,\mu)$ a measurable cocycle, namely $\alpha(a,b)\alpha(b,c)=\alpha(a,c)$ holds for a.e. $a$ and every $b,c$ with $a\cR b \cR c$. We think of $\cR$ as analogous to a group, $\Omega$ as analogous to a family of probability measures on a group and $\alpha$ as analogous to a group action. Given a measurable function $F$ on $B\times X$, define the averages $\A[F|\omega_i]$ on $B\times X$ by
$$\A[F|\omega_i](b,x) = \sum_{c \in B} \omega_i(c,b) F(c, \alpha(c,b)^{-1}x).$$
Let $\cR_\alpha$ be the skew-product relation on $\cR \times X$ defined by $(b,x)\cR_\alpha (c,y) \Leftrightarrow b\cR c$ and $\alpha(c,b)x=y$. We say that $\Omega$ is a pointwise ergodic family in $L^p$ if for every $F \in L^p(B\times X)$, $\A[F|\omega_i]$  converges pointwise a.e. as $i\to\infty$ to $\EE[F | \cR_\alpha]$, the conditional expectation of $F$ on the sigma-algebra of $\cR_\alpha$-saturated sets (i.e. those measurable sets which are unions of $\cR_\alpha$-classes modulo null sets). 


\subsubsection{Finding pointwise ergodic families} One natural problem that arises is how to find collections $\Omega$ as above. In case $\cR$ is amenable, several criteria for establishing that a given $\Omega$ is a pointwise ergodic family were developed in \cite{BN2}. These criteria are based on a generalization of the notions of doubling, regular, and tempered F\o lner sets in the group to F\o lner subset functions on the equivalence relation. In the case of tempered subset function, for example, \cite{BN2} generalizes  Lindenstrauss' pointwise ergodic theorem for tempered Folner sequences on an amenable group \cite{Li01}. 
The fact that such an $\Omega$ always exists follows from the fact that $\cR$ can be realized as the orbit-equivalence relation for a $\ZZ$-action \cite{CFW}. It should be noted, however, that this is an abstract existence result and finding an explicit system of probability measures $\Omega$ depends on having an explicit presentation for the relation $\cR$.  


\subsubsection{Using cocycles to obtain pointwise ergodic theorems}

Suppose now that we have a measured equivalence relation $\cR$ as above with a pointwise ergodic family $\Omega$. Given a measurable cocycle $\alpha:\cR \to \Gamma$ (where $\Gamma$ is a countable group) we can push $\Omega$ forward to obtain a family of measures on $\Gamma$. But before doing so, let us consider a measure-preserving action of $\Gamma$ on a standard probability space $(X,\mu)$ given by a homomorphism $\beta:\Gamma \to \Aut(X,\mu)$. Then the composition $\beta\alpha:\cR \to \Aut(X,\mu)$ is a cocycle. If $f$ is a function on $X$ and $F$ is defined on $B\times X$ by $F(b,x)=f(x)$ then
$$\int \A[F|\omega_i](b,x)~d\nu(b) = \int \sum_{c \in B} \omega_i(c,b) f(\beta(\alpha(c,b))^{-1}x)~d\nu(b) = \sum_{\gamma \in \Gamma} \omega^\alpha_i(\gamma) f(\beta(\gamma)^{-1} x)$$
where $\omega^\alpha_i(\gamma) =  \int \sum_{c \in B:~\alpha(c,b)=\gamma} \omega_i(c,b) ~d\nu(b)$. From here we would like to conclude that $\pi_X(\omega^\alpha_i)(f)$ converges pointwise a.e. as $i\to\infty$ to $\int \EE[F | \cR_{\beta\alpha}](b,\cdot)~d\nu(b)$. In Theorem \ref{thm:general} below, we obtain this conclusion if $\Omega$ satisfies a strong $L^p$-maximal inequality which allows us to interchange the integral with the limit, $\Omega$ is pointwise ergodic in $L^p$ and $f \in L^p(X)$.

While this shows that $\{\omega^\alpha_i\}_{i\in \II}$ is a pointwise convergent family of measures on $\Gamma$, it does not establish that is pointwise ergodic, because we do not know apriori whether the limit $\int \EE[F| \cR_{\beta\alpha}](b,\cdot)~d\nu(b)$ coincides with the conditional expectation of $f$ on the sigma-algebra of $\Gamma$-invariant measurable sets. We say that the cocycle $\alpha:\cR \to \Gamma$ is {\em weakly mixing} if whenever $\beta:\Gamma \to \Aut(X,\mu)$ is an ergodic action then the skew-product relation $\cR_{\beta\alpha}$ on $B\times X$ is also ergodic. In this case, we can conclude that $\EE[F| \cR_{\beta\alpha}](b,x)= \EE[f|\Gamma](x)$ almost everywhere and therefore $\{\omega^\alpha_i\}_{i\in \II}$ is a pointwise ergodic family as desired. Thus our goal now is to find, for a given group $\Gamma$, a weakly mixing cocycle $\alpha$ from an amenable pmp equivalence relation $\cR$ into $\Gamma$ together an appropriate family $\Omega$ of leafwise measures on $\cR$. As we will see, the weakly mixing criterion can be  weakened somewhat, a fact that will be important below.


\subsubsection{Equivalence relations from amenable actions: the $II_1$ case} Given a group $\Gamma$, we can construct an amenable pmp equivalence relation $\cR$ with a cocycle $\alpha:\cR \to \Gamma$ from an essentially free amenable action of $\Gamma$. The simplest case occurs when $\Gamma$ admits an essentially free measure-preserving amenable action on a standard probability space $(B,\nu)$. Let $\cR=\{(b,gb):~g\in \Gamma\}$ be the orbit-equivalence relation and $\alpha:\cR \to \Gamma$ be the cocycle $\alpha(gb,b)=g$ which is well-defined a.e. because the action is essentially free. If the action is weakly mixing then $\alpha$ is also weakly mixing.  Of course, if $\Gamma$ admits such an action then it must be an amenable group and we can construct the family $\Omega$ from a regular or  tempered F\o lner sequence of $\Gamma$. From this construction and \cite[Theorem 2.5]{BN2}, we deduce  the pointwise ergodic theorem for regular or tempered F\o lner sequences, in the latter case using an argument modelled after B. Weiss' proof of Lindenstrauss' Theorem \cite{We03}. 

\subsubsection{The $II_\infty$ case} The next simplest case occurs when $\Gamma$ admits an essentially free measure-preserving amenable action on a sigma-finite (but infinite) measure space $(B',\nu)$. In this case, let $B \subset B'$ be a measurable subset with $\nu(B)=1$. Let $\cR=\{(b,gb):~g\in \Gamma, b \in B\} \cap (B \times B)$ be the orbit-equivalence relation and $\alpha:\cR \to \Gamma$ be the cocycle $\alpha(gb,b)=g$ which is well-defined a.e. because the action is essentially free. If the action is weakly mixing then $\alpha$ is also weakly mixing. Because the action is amenable, $\cR$ is amenable and so pointwise ergodic families $\Omega$ of $\cR$ do exist although we do not have an explicit general recipe for finding them. 

\subsubsection{The $III$-case and the ratio set} Perhaps the most important case occurs when $\Gamma \cc (\bar{B},\bar{\nu})$ is an essentially free amenable action but $\bar{\nu}$ is not equivalent to an invariant $\sigma$-finite measure. In this case we form the Maharam extension which is an amenable measure-preserving action on a sigma-finite measure space. The details of this construction depend very much on the essential range (also known as the ratio set) of the Radon-Nikodym derivative. To be precise, we say that a real number $t\ge 0$ is in the ratio set if for every $\epsilon>0$ and every subset $A \subset \bar{B}$ with positive measure there exists a subset $A' \subset A$ with positive measure and an element $g \in \Gamma \setminus \{e\}$ such that $gA' \subset A$ and $|\frac{d\bar{\nu}\circ g}{d\bar{\nu}}(a) -t| < \epsilon$ for every $a\in A'$. A similar definition holds for stating that $+\infty$ is in the ratio set. The action is said to be of type $II$ if the ratio set is $\{1\}$, it is type $III_1$ if the ratio set is $[0,\infty]$, it is type $III_\lambda$ if the ratio set is $\{\lambda^n:~n\in \Z\}\cup\{0,\infty\}$ (where $\lambda \in (0,1)$) and it is type $III_0$ if the ratio set is $\{0,1,\infty\}$. There are no other possibilities. It is also known that, after replacing $\bar{\nu}$ with an equivalent measure, we may assume that for every $g \in \Gamma$, $\frac{d\bar{\nu} \circ g}{d\bar{\nu}}(b)$ is contained in the ratio set for a.e. $b$.  We say that the action $\Gamma \cc (\bar{B},\bar{\nu})$ has {\em stable type $III_\lambda$} (for some $\lambda \in (0,1)$) if for every ergodic measure-preserving action $\Gamma \cc (X,\mu)$ the ratio set of the product action $\Gamma \cc (\bar{B}\times X, \bar{\nu}\times \mu)$ contains $\{\lambda^n:~n\in \Z\}$ and $\lambda \in (0,1)$ is the largest number with this property. Similarly, the action of $\Gamma$ on $(\bar{B}, \bar{\nu})$ has stable type $III_1$ if the ratio set of the product action is $[0,\infty]$ (for every pmp action $\Gamma \cc (X,\mu)$). For more details, see \S \ref{sec:type}.

\subsubsection{The $III_1$ case} Suppose that $\Gamma \cc (\bar{B},\bar{\nu})$ has type $III_1$.  Let  
$$R(g,b) = \log\left( \frac{d\bar{\nu} \circ g}{d\nu}(b) \right).$$ 
Let $B' = \bar{B} \times \RR$ and $\nu' = \bar{\nu} \times \theta$ where $\theta$ is the measure on $\R$ defined by $d\theta(t)=e^tdt$. The group $\Gamma$ acts on $(B',\nu')$ by $g(b,t)=(gb,t-R(g,b))$. This action preserves the measure $\nu'$, which is infinite. So we can let $B=\bar{B}\times [0,T] \subset B'$ (for some $T>0$), $\nu$ equal to the restriction of $\nu'$ of $B$ and $\cR$ be the orbit-equivalence relation on $B$ given by: $(b,t)\cR(b',t')$ if there exists $g \in \Gamma$ such that $g(b,t)=(b',t')$. Also let $\alpha:\cR \to \Gamma$ be the cocycle $\alpha(gb,b)=g$. If the action $\Gamma \cc (\bar{B},\bar{\nu})$ is weakly mixing and stable type $III_1$ then $\alpha$ is weakly mixing. For example, these conditions are satisfied when $\Gamma$ is an irreducible lattice in a connected semi-simple Lie group  $G$ without compact factors and $\bar{B}$ is the homogeneous space $G/P$ where $P<G$ is a minimal parabolic subgroup (see \cite{BN2} for details). More generally, if $\Gamma \cc (\bar{B},\bar{\nu})$ is weakly mixing, type $III_1$ and stable type $III_\lambda$ for some $\lambda \in (0,1]$ then in Theorem \ref{thm:typeIII} we prove that $\alpha$ is weakly mixing relative to a compact group action on $B$ (Definition \ref{defn:compact}). This slightly weaker condition is enough to obtain a pointwise ergodic family on $\Gamma$ from one on $\cR$ (Theorem \ref{thm:general}).

\subsubsection{The $III_\lambda$ case} Suppose, as above, that $\Gamma \cc (\bar{B},\bar{\nu})$ is an essentially free amenable action and $\bar{\nu}$ is not equivalent to an invariant $\sigma$-finite measure. For $\lambda \in (0,1)$ let 
$$R_\lambda(g,b) = \log_\lambda\left( \frac{d\bar{\nu} \circ g}{d\nu}(g) \right).$$
Suppose that $R_\lambda(g,b) \in \Z$ for every $g\in \Gamma$ and a.e. $b$. In this case, we form the discrete Maharam extension. To be precise, let $B' = \bar{B} \times \Z$ and $\nu'  = \bar{\nu} \times \theta_\lambda$ where $\theta_\lambda(\{n\})=\lambda^{-n}$. Let $\Gamma$ act on $B'$ by $g(b,t)=(g,b+R_\lambda(g,b))$. This action preserves $\nu'$ and we can define $B, \cR$ and $\alpha$ as in the previous case. If $\Gamma \cc (\bar{B},\bar{\nu})$ is weakly mixing and stable type $III_\lambda$ then $\alpha$ is also weakly mixing. More generally, if $\Gamma \cc (\bar{B},\bar{\nu})$ is weakly mixing, type $III_\lambda$ and stable type $III_\tau$ for some $\tau \in (0,1)$ then $\alpha$ is also weakly mixing relative to a certain compact group action (Theorem \ref{previous}) which, as we see from Theorem \ref{thm:general} below, is sufficient for constructing pointwise ergodic families. 

\subsubsection{Amenable actions: examples}
We now have the problem of finding weakly mixing cocycles from amenable equivalence relations into $\Gamma$. Theorem \ref{thm:existence} shows that such cocycles exist for any countable group $\Gamma$. However, it is desirable to have such equivalence relations which arise from actions of $\Gamma$ to determine pointwise ergodic theorems in which the family of measures is associated with some geometric structure on the group. Next we discuss a few possibilities.

The first possibility is to choose $(\bar{B},\bar{\nu})$ to be the Poisson boundary $B(\Gamma,\eta)$ with respect to an admissible probability measure $\eta$ on $\Gamma$. It is known that the action $\Gamma$ on $B(\Gamma,\eta)$ is amenable and weakly mixing (in fact doubly-ergodic \cite{Ka}). However, the type and stable type of the action is not well understood in general. There are some exceptions: for example  in \cite{INO08} it is proven that the Poisson boundary of a random walk on a Gromov hyperbolic group induced by a nondegenerate measure on $\Gamma$ of finite support is never of type $III_0$. 

Another possibility is to choose $(\bar{B},\bar{\nu})$ to be the square of the Poisson boundary $B(\Gamma,\eta)$. Because $\Gamma \cc B(\Gamma,\eta)$ is amenable and doubly-ergodic \cite{Ka}, this action is amenable and weakly mixing. The type and stable type of this action is not well-understood in general but there are many cases in which it is known to be of type $II_\infty$ (this means that there is an equivalent invariant infinite $\sigma$-finite measure).

If the group $\Gamma$ is word hyperbolic then we may consider the action of $\Gamma$ on its boundary with respect to a Patterson-Sullivan measure. It is well-known that this action is amenable but it is not known whether it is necessarily weakly mixing. In \cite{B2} it is shown that the stable ratio set of this action contains a real number $t \notin \{0,1\}$. It follows that if this action is weakly mixing then it has stable type $III_\lambda$ for some $\lambda \in (0,1]$ and therefore, the methods of this paper apply. It is also known that the square of the Patterson-Sullivan measure on the squared boundary $\partial \Gamma \times \partial \Gamma$ is equivalent to an invariant sigma-finite measure. 

If the group $\Gamma$ is an irreducible lattice in a connected semi-simple Lie group $G$ which has trivial center and no compact factors, then we can consider the action of $\Gamma$ on $G/P$ (where $P<G$ is a minimal parabolic subgroup) with respect to a $G$-quasi-invariant measure. In \cite{BN2} it is shown that this action is weakly mixing and has type and stable type $III_1$.

{\bf Organization}.
In \S \ref{sec:mer} we review pointwise ergodic theory for measured equivalence relations. In \S \ref{sec:construct} we prove our first main result, Theorem \ref{thm:general}, which is a general construction of pointwise ergodic families. \S \ref{sec:amenable} reviews type, stable type and the Maharam extension. In \S \ref{sec:typeIII} we prove our second main result, Theorem \ref{thm:typeIII}, which explains how to produce a weakly mixing cocycle relative to a compact group action from a sufficiently nice type $III_1$ action. In \S \ref{sec:example} we illustrate our approach with a variety of examples.

 \section{Measured equivalence relations}\label{sec:mer}
Before discussing the construction of pointwise ergodic families for $\Gamma$-actions,  we must first introduce the notion of a pointwise ergodic family for a measured equivalence relation. Let $(B,\nu)$ be a standard Borel probability space and $\cR \subset B\times B$ a measurable equivalence relation. Recall that this means that there exists a co-null Borel set $B_0\subset B$ such that  the restriction of $\cR$ to $B_0\times B_0$ is a Borel equivalence relation. We assume that  $\cR$ is discrete, namely the equivalence classes are countable, almost always, namely on a conull invariant Borel subset $B_0$ . We assume as well that $\nu$ is $\cR$-invariant, i.e., if $\phi:B\to B$ is a measurable automorphism with graph contained in $\cR$ (when restricted to a conull invariant Borel set)  then $\phi_*\nu=\nu$.

\subsection{Leafwise pointwise ergodic families on $B$.}

Let $\II\in \{\RR_{>0},\NN\}$ be an index set and $\Omega=\{\omega_i\}_{i\in \II}$ a measurable collection of  leafwise probability measures $\omega_i:\cR \to [0,1]$. More precisely, we assume that each $\omega_i$ is defined on a common conull invariant Borel set $B_0$, where the map $(i, (b,c)) \in \II \times \cR \to \omega_i(b,c) \in [0,1]$ is Borel and for  every $b\in B_0$, $\sum_{c \in B_0} \omega_i(c,b)=1$ where, for convenience, we set $\omega_i(c,b)=0$ if $(c,b) \notin \cR$. 


Let $\alpha:\cR \to \Aut(X,\mu)$ be a measurable cocycle into the group of measure-preserving automorphisms of a standard Borel probability space $(X,\mu)$. This means that $\alpha(b,c)\alpha(c,d)=\alpha(b,d)$ for every $b,c,d \in B_0$ with $(b,c), (c,d) \in \cR$, and $B_0$ a conull invariant Borel set. Let $\cR_\alpha$ be the equivalence relation on $B_0\times X$ given by $(b,x)\cR_\alpha (c,\alpha(c,b)x)$ (for $(b,c) \in \cR, x\in X$). Given a Borel function $F_0$ on $B_0\times X$, we define its $\Omega$-averages, $\A[F_0|\omega_i]$ as a Borel function on $B_0\times X$ 
$$\A[F_0|\omega_i](b,x) := \sum_{c\in B_0} \omega_i(c,b) F_0(c, \alpha(c,b)^{-1}x)$$
when this sum is absolutely convergent. We say that $\Omega$ is a {\em pointwise ergodic family in $L^p$} if for every such cocycle $\alpha$, for every $F \in  L^p(B\times X,\nu\times \mu)$, choosing some Borel restriction of $F_0$ to $B_0\times X$, 
$$\lim_{i\to\infty} \A[F_0|\omega_i](b,x) = \EE[F | \cR_\alpha ](b,x)$$
for a.e. $(b,x) \in B\times X$ where $\EE[F| \cR_\alpha]$ denote the conditional expectation of $F$ on the $\sigma$-algebra of $\cR_\alpha$-saturated Borel sets (recall that a set $A \subset B\times X$ is $\cR_\alpha$-saturated if it is a union of $\cR_\alpha$-equivalence classes, up to a null set). Our definition is of course independent of the choice of $F_0$ in the $L^p$-equivalence class of $F$. 

We will also need to make use of leafwise maximal inequalities. For this purpose, let $\MM[F_0| \Omega] : = \sup_{i\in \II} \A[|F_0| |\omega_i]$, where $F_0$ is Borel on $B_0\times X$. We say that $\Omega$ satisfies the {\em strong $L^p$ maximal inequality} if 
$\MM[F_0| \Omega] $ is a measurable function, and there is a constant $C_p>0$ such that $\| \MM[F_0|\Omega] \|_p \le C_p \|F\|_p$ for every $F \in L^p(B\times X,\nu\times \mu)$, and for every cocycle $\alpha$ as above. 
We say that $\Omega$ satisfies the {\em weak (1,1)-type maximal inequality} if  there is a constant $C_{1,1}$ such that for every $t>0$ and $F \in L^1(B\times X,\nu\times\mu)$,
$$\nu\times \mu(\{ (b,x):~ \M[F_0|\Omega](b,x)\ge t\}) \le \frac{ C_{1,1} \|F\|_1}{t}.$$
Our definitions are of course independent of the choice of $F_0$ in the $L^p$-equivalence class of $F$. 

\begin{remark}
In the rest of the paper we will  resort to the standard practice of considering just the space $B$, with the relation $\cR$, the cocycle $\alpha$ and the averaging family $\Omega$ being defined and Borel on some unspecified invariant conull Borel set $B_0$. The  statements and proofs of the results that we discuss should be interpreted as holding on an (unspecified)  conull invariant Borel set, which is abbreviated by saying that the desired property  holds almost always, or almost surely, or essentially.  
\end{remark}

\subsection{Ergodicity of the extended relation : weak mixing.}
Now suppose $\Gamma$ is a countable group, $\cR\subset B\times B$ is a discrete measurable equivalence relation,  and $\alpha:\cR \to \Gamma$ a measurable cocycle (so $\alpha(a,b)\alpha(b,c)=\alpha(a,c)$ almost everywhere). Given a homomorphism $\beta:\Gamma \to \Aut(X,\mu)$ we can consider $\beta\alpha$ as a cocycle into $\Aut(X,\mu)$. Suppose we can construct a measurable leafwise systems of probability measures $\Omega$ on $B$ with good averaging properties, namely such that we can establish maximal or pointwise ergodic theorems for the averaging operators $\A[\cdot|\omega_i]$ defined above. It is then natural to average these operators over $B$, and thus obtain operators defined on functions on $X$, which are given by averaging over probability measures on $\Gamma$. 

This is indeed the path we plan to follow, but note that there is an additional unavoidable difficulty: the equivalence relation $\cR_{\beta\alpha}$ is not necessarily ergodic even if $\Gamma \cc^\beta (X,\mu)$ is an ergodic action. Because of this problem, even if we prove pointwise convergence, the limit function might not be invariant. To resolve this problem, we need to make additional assumptions on the cocycle, as explained next.

\begin{defn}[Weakly mixing]
We say that a measurable cocycle $\alpha:\cR \to \Gamma$ is {\em weakly mixing} if for every ergodic pmp action $\Gamma \cc^\beta (X,\mu)$ the induced equivalence relation $\cR_{\beta \alpha}$ on $B\times X$ is ergodic.
\end{defn}

\begin{defn}[Weakly mixing relative to a compact group action]\label{defn:compact}
Let $K$ be a compact group with a nonsingular measurable action $K \cc (B,\nu)$. We say that a measurable cocycle $\alpha:\cR \to \Gamma$ is {\em weakly mixing relative to the action $K\cc (B,\nu)$} if for every pmp action $\Gamma \cc^\beta (X,\mu)$ and every $f \in L^1(X)\subset L^1(B\times X)$, 
$$\int \EE[f| {\cR_{\beta \alpha}}] (kb,x)~dk  = \EE[f | \Gamma](x)$$
for a.e. $(b,x)$ where $dk$ denotes Haar probability measure on $K$, $\EE[f| {\cR_{\beta \alpha}}]$ is the conditional expectation of $f$ (viewed as an element of $L^1(B\times X)$) on the $\sigma$-algebra of $\cR_{\beta \alpha}$-saturated measurable sets and $\EE[f| \Gamma]$ is the conditional expectation of $f$ on the $\sigma$-algebra of $\Gamma$-invariant sets. For example, if $K$ is the trivial group then this condition implies $\alpha$ is weakly mixing.
\end{defn}

\begin{defn}
We say that an action $K \cc (B,\nu)$ as above has {\em uniformly bounded RN-derivatives} if there is a constant $C(K)$ such that 
$$ \frac{d\nu \circ k}{d\nu}(b) \le C(K)\quad \textrm{ for a.e. } (k,b) \in K \times B.$$
\end{defn}

The condition of weakly-mixing relative to a compact group action is both useful and natural and can be verified in practice in many important situations. The reason for that is its close connection to the notion of type of a non-singular group action, as we will explain further below.  

\section{Construction of ergodic averages}\label{sec:construct}

\subsection{Statement of Theorem \ref{thm:general}}

We can now state the main result on the construction of pointwise ergodic families of probability measures on $\Gamma$. It reduces it to the construction of  a system of leafwise probability measures $\Omega$ with good averaging properties on an equivalence relation $\cR$ on $B$, provided we also have a weakly-mixing cocycle $\alpha : \cR \to \Gamma$ (relative to a compact group action).

\begin{thm}\label{thm:general}
Let $(B,\nu,\cR)$ be a measured equivalence relation, $\Omega=\{\omega_i\}_{i\in \II}$ a measurable family of leafwise probability measures $\omega_i:\cR \to [0,1]$ and $\alpha:\cR \to \Gamma$ be a measurable cocycle into a countable group $\Gamma$. Suppose there is a nonsingular compact group action $K \cc (B,\nu)$ with uniformly bounded RN-derivatives and $\psi\in L^q(B,\nu)$ ($1<q<\infty$) is a probability density (so $\psi \ge 0$, $\int \psi~d\nu=1$). For $i \in \II$, define the probability measure $\zeta_i$ on $\Gamma$ by
$$\zeta_i(\gamma):= \int_B \int_K \sum_{c:~\alpha(c,b)=\gamma} \omega_i(c,kb) \psi(b) ~dkd\nu(b).$$
Let $p>1$ be such that $\frac{1}{p}+\frac{1}{q}=1$. 
Then the following hold.
\begin{enumerate}
\item If $\Omega$ satisfies the strong $L^p$ maximal inequality then $\{\zeta_i\}_{i\in\II}$ also satisfies the strong $L^p$ maximal inequality.
\item If in addition, $\Omega$ is a pointwise ergodic family in $L^p$, then $\{\zeta_i\}_{i\in\II}$ is a pointwise convergent family in $L^p$.
\item If in addition, $\alpha$ is weakly mixing relative to the $K$-action, then $\{\zeta_i\}_{i\in \II}$ is a pointwise ergodic family in $L^p$.
\end{enumerate}
Similarly, if $\Omega$ satisfies the weak $(1,1)$-type maximal inequality and $\psi\in L^\infty(B,\nu)$ then $\{\zeta_r\}_{r\in \II}$ satisfies the $L \log (L)$ maximal inequality. If in addition, $\Omega$ is a pointwise ergodic family in $L^1$, then $\{\zeta_i\}_{i\in\II}$ is a pointwise convergent family in $L\log(L)$. If in addition, $\alpha$ is weakly mixing relative to the $K$-action, then $\{\zeta_i\}_{i\in \II}$ is a pointwise ergodic family in $L\log(L)$.
\end{thm}



\subsection{Proof of Theorem \ref{thm:general}}
In this section, we prove Theorem \ref{thm:general}, motivated by the proof of \cite[Theorem 4.2]{BN2}.  Let $\Gamma$ be a countable group, $(B,\nu,\cR)$ be a pmp equivalence relation with a leafwise system $\Omega$ of probability measures on the equivalence classes of $\cR$. Let $\alpha:\cR \to \Gamma$ be a measurable cocycle, $K \cc (B,\nu)$ a nonsingular action of a compact group (with uniformly bounded RN-derivatives), and $\psi \in L^q(B,\nu)$ be a probability density. Let $\Gamma \cc^\beta (X,\mu)$ be a pmp action and $f \in L^p(X,\mu)$. For convenience, for $g\in \Gamma$ and $x\in X$ we write $gx=\beta(g)x$. We consider $f$ to also be a function of $B\times X$ defined by $f(b,x)=f(x)$. Observe that by definition of the probability measures $\zeta_i$ given in Theorem \ref{thm:general} : 
\begin{eqnarray}
\pi_X(\zeta_i)(f)(x)&=& \sum_{\gamma \in \Gamma} \zeta_i(\gamma) f(\gamma^{-1} x)\\
&=& \sum_{\gamma \in \Gamma} f(\gamma^{-1} x) \int_B\int_K\sum_{c:~\alpha(c,b)=\gamma} \omega_i(c,kb) \psi(b) ~dkd\nu(b)\\
&=& \int_B\int_K \sum_{c:~(c,b) \in \cR} \omega_i(c,kb) f(c,\alpha(c,kb)^{-1}x)\psi(b) ~dkd\nu(b)\\
&=& \int_B\int_K \A[f | \omega_i](kb,x)\psi(b)~dk d\nu(b). \label{eqn:1}
\end{eqnarray}

Because of this equality, we extend the domain of the operator $\pi_X(\zeta_i)$ from $L^p(X)$ to $L^p(B\times X)$ by setting
\begin{eqnarray}\label{transform}
\Pi(\zeta_i)(F)(x) = \int_B\int_K \A[F| \omega_i](kb,x)\psi(b)~dk d\nu(b)
\end{eqnarray}
for any $F \in L^p(B\times X)$. 
Thus $\Pi(\zeta_i)$ is an operator from $L^p(B\times X)$ to $L^p(X)$, and $\Pi(\zeta_i)=\pi_X(\zeta_i)$ on the space $L^p(X)$, viewed as the subspace of $L^p(B\times X)$ consisting of functions independent of $b\in B$.

Let us also set $\sM[F|\zeta] = \sup_{i\in \II} \Pi(\zeta_i)(|F|)$ for any $F\in L^p(B\times X)$, and begin the proof by establishing maximal inequalities, as follows. 

\begin{thm}\label{thm:maximal}

If $\Omega$ satisfies the strong $L^p$ maximal inequality and $\frac{1}{p} + \frac{1}{q} =1$ then there is a constant $\bar{C}_p>0$ such that for any $F\in L^p(B\times X)$,
$$\| \sM[F|\zeta] \|_p \le \bar{C}_p \| F\|_p.$$
If $\Omega$ satisfies the $L\log(L)$-maximal inequality and $\psi\in L^\infty(B,\nu)$, then there is a constant $\bar{C}_1>0$ such that for any $F\in L\log L(B\times X)$,
$$\| \sM[F|\zeta] \|_{1} \le \bar{C}_1 \|F\|_{L\log L}.$$
\end{thm}


\begin{proof}
Without loss of generality, we may assume $F \ge 0$. Let us first consider the case $p>1$ and $\frac{1}{p}+\frac{1}{q}=1$. Because $\Omega$ satisfies the maximal inequality for functions in $L^p(B\times X)$, there is a constant $C_p>0$ such that $\| \sM[F | \Omega] \|_p \le C_p \|F \|_p$  for every $F\in L^p(B\times X)$. 
By (\ref{transform}) and H\"older's inequality,
\begin{eqnarray*}
\|\sM[F|\zeta]\|_p^p &=& \int_X \left| \sup_{i\in \II} \int_B\int_K \A[F | \omega_i](kb,x)\psi(b)~dk d\nu(b) \right|^p~d\mu(x)\\
&\le & \int_X  \sup_{i\in \II}\left(  \int_B\int_K \A[F | \omega_i](kb,x)^p~dk d\nu(b)\right)\left( \int_B\int_K \psi(b)^q~dkd\nu(b)\right)^{p/q}~d\mu(x)\\
&=& \| \psi\|_q^p \int_X  \sup_{i\in \II} \int_B\int_K \A[F | \omega_i](kb,x)^p~dk d\nu(b)~d\mu(x)\\
&\le&\| \psi\|_q^p  \int_X\int_B\int_K \sM[F | \Omega](kb,x)^p~dk d\nu(b)d\mu(x)\\
&=& \| \psi\|_q^p  \int_X\int_B\int_K \sM[F| \Omega](b,x)^p \frac{d\nu \circ k^{-1}}{d\nu}(b)~dk d\nu(b)d\mu(x)\\
&\le& C(K) \| \psi\|_q^p \int_X\int_B \sM[F | \Omega](b,x)^p ~ d\nu(b)d\mu(x) = C(K) \| \psi\|_q^p \| \sM[F | \Omega] \|_p^p\\
&\le& C(K) \| \psi\|_q^p C_p^p \| F\|^p_p = \bar{C}_p \| F\|_p.
\end{eqnarray*}

As to the $L\log L$ results, let us now suppose $F \in L\log L(B \times X)$ and $\psi \in L^\infty(B)$. We assume there is a constant $C_1>0$ (independent of $F$ and the action $\Gamma \cc (X,\mu)$) such that $\|\sM[F|\Omega]\|_1 \le C_1\|F\|_{L\log L} = C_1 \|F\|_{L\log L}$ for any $F\in L\log L(B\times X)$. The proof that $\| \sM[F | \zeta] \|_1 \le C(K)C_1 \|\psi\|_\infty \|F\|_{L\log L}$ is now similar to the proof of the $p>1$ case above.
\end{proof}

\begin{proof}[Proof of Theorem \ref{thm:general}]
Theorem \ref{thm:maximal} proves the first conclusion, namely the maximal inequalities. To obtain the second conclusion, namely pointwise convergence, without loss of generality we may assume $\Gamma \cc (X,\mu)$ is ergodic. Suppose now that $F \in L^\infty(B\times X)$. The bounded convergence theorem and the assumption that $\Omega$ is a pointwise ergodic family implies that for a.e. $x\in X$,
\begin{eqnarray*}
\lim_{r\to\infty} \Pi( \zeta_r)(F)(x) &=& \lim_{r\to\infty}  \int_B\int_K \sA[F|\omega_r](kb,x)\psi(b)~dkd\nu(b)\\
&=&  \int_B\int_K \lim_{r\to\infty}  \sA[F|\omega_r](kb,x)\psi(b)~dkd\nu(b)\\
&=&   \int_B\int_K   \EE[F|{\cR_{\beta\alpha}}](kb,x)\psi(b)~dkd\nu(b).
\end{eqnarray*}
In particular, $\{\zeta_r\}_{r\in \II}$ is a pointwise convergent sequence in $L^\infty(X) \subset  L^\infty(B\times X)$. 

Suppose now that $F \in L^p(B\times X)$ for some $p>1$ and $\frac{1}{p}+\frac{1}{q}=1$. We will show that for a.e. $x\in X$, 
$$\lim_{r\to\infty} \Pi( \zeta_r)(F)(x) =  \int_B\int_K   \EE[F|{\cR_{\beta\alpha}}](kb,x)\psi(b)~dkd\nu(b).$$
By replacing $F$ with $F - \EE[F|{\cR_{\beta\alpha}}]$ if necessary, we may assume $\EE[F|{\cR_{\beta\alpha}}]=0$. 

Let $\epsilon>0$. Because $L^\infty(B\times X)$ is dense in $L^p(B\times X)$, there exists an element $F' \in L^\infty(B\times X)$ such that $\|F - F'\|_p \le \epsilon$ and $\EE[F'|{\cR_{\beta\alpha}}]=0$. Then for a.e. $x\in X$,
\begin{eqnarray*}
\limsup_{r\to\infty} \left|  \Pi( \zeta_r)(F)(x) \right|&\le & 
\limsup_{r\to\infty}  \left| \Pi(\zeta_r)(F-F')(x)\right| +  \lim_{r\to\infty} \left|\Pi(\zeta_r)(F')(x)\right| \\
&=&\limsup_{r\to\infty}  \left| \Pi(\zeta_r)(F-F')(x)\right| \le  \sM[F - F'|\zeta](x).
\end{eqnarray*}
Thus if $\tilde{F}:=\limsup_{r\to\infty} \left| \Pi(\zeta_r)(F)\right|$, then 
$$\|\tilde{F}\|_p \le \| \sM[F - F'|\zeta] \|_p \le C'_p  \| F-F'\|_p \le  C'_p \epsilon  $$
for some constant $C'_p>0$ by Theorem \ref{thm:maximal}. Since $\epsilon$ is arbitrary, $\|\tilde{F}\|_p = 0$ which implies 
$$\lim_{r\to\infty} \Pi( \zeta_r)(F)(x) = 0= \int_B\int_K   \EE[F|{\cR_{\beta\alpha}}](kb,x)\psi(b)~dkd\nu(b)$$
for a.e. $x$ as required. This proves $\{\zeta_r\}_{r\in \II}$ is a pointwise convergent sequence in $L^p(X)\subset L^p(B\times X)$. 

Finally, to prove the third claim in Theorem  \ref{thm:general}, assume $\alpha$ is weakly mixing relative to the $K$-action. Then  given $f\in L^p(X)$, because $\iint   \EE[f|{\cR_\alpha}](kb,x)\psi(b)~dkd\nu(b) = \EE[f|\Gamma](x)$, $\{\zeta_r\}_{r\in \II}$ is a pointwise ergodic sequence in $L^p$. The proof of the case when $f \in L\log L(X)$ and $\psi \in L^\infty(B)$ is similar.

\end{proof}

\section{Amenable actions}\label{sec:amenable}

As noted already, the notion of weakly-mixing of a cocycle relative to a compact group arises naturally when we consider the type of an amenable action, which arises in our context as the obstruction to the ergodicity of the extended relation.  
We now proceed to define and discuss the fundamental notions of type and stable type of a non-singular action. 

\subsection{The ratio set, type, and stable type of a non-singular action}\label{sec:type}
Let $\Gamma$ be a countable group and $\Gamma \cc (B,\nu)$ a nonsingular action on a standard probability space. The {\em ratio set} of the action, denoted $\RS(\Gamma \cc (B,\nu)) \subset [0,\infty]$, is defined as follows. A real number $r\ge 0$ is in $\RS(\Gamma \cc (B,\nu))$ if and only if for every positive measure set $A \subset B$ and $\epsilon>0$ there is a subset $A' \subset A$ of positive measure and an element $g\in \Gamma \setminus\{e\}$ such that 
\begin{itemize}
\item $gA' \subset A$,
\item $| \frac{d\nu \circ g}{d\nu}(b)-r| < \epsilon$ for every $b \in A'$.
\end{itemize}
The extended real number $+\infty \in \RS(\Gamma \cc (B,\nu))$ if and only if for every positive measure set $A \subset B$ and $n>0$ there is a subset $A' \subset A$ of positive measure and an element $g\in \Gamma \setminus\{e\}$ such that 
\begin{itemize}
\item $gA' \subset A$,
\item $ \frac{d\nu \circ g}{d\nu}(b) > n$ for every $b \in A'$.
\end{itemize}
The ratio set is also called the {\em asymptotic range} or {\em asymptotic ratio set}. By Proposition 8.5 of \cite{FM77}, if the action $\Gamma \cc (B,\nu)$ is ergodic then $\RS(\Gamma \cc (B,\nu))$ is a closed subset of $[0,\infty]$. Moreover, $\RS(\Gamma \cc (B,\nu)) \setminus \{0,\infty\}$ is a multiplicative subgroup of $\RR_{>0}$. Since 
$$\frac{d\nu \circ g^{-1}}{d\nu}(gb) = \left( \frac{d\nu \circ g}{d\nu}(b) \right)^{-1},$$
the number $0$ is in the ratio set if and only if $\infty$ is in the ratio set. So if $\Gamma \cc (B,\nu)$ is ergodic and non-atomic then the possibilities for the ratio set and the corresponding type classification are:
\begin{displaymath}
\begin{array}{c|cc}
\textrm{ ratio set } & \textrm{ type }\\
\hline
\{1\} & II &\\
\{0,1,\infty\} & III_0&\\
\{0,\lambda^n, \infty: ~n\in \ZZ\} & III_\lambda & (0 < \lambda < 1)\\
 \lbrack 0,\infty \rbrack & III_1&
 \end{array}\end{displaymath}
For a very readable review, see \cite{KW91}. There is also an extension to general cocycles taking values in an arbitrary locally compact group in section 8 of \cite{FM77}. 

Observe that if $\Gamma \cc (X,\mu)$ is a pmp action then the ratio set of the product action satisfies $\RS(\Gamma \cc (B\times X, \nu \times \mu)) \subset \RS(\Gamma \cc (B,\nu))$. Therefore, it makes sense to define the {\em stable} ratio set of $\Gamma \cc (B,\nu)$ by $\SRS(\Gamma \cc (B,\nu)) = \cap \RS(\Gamma \cc ( B\times X,\nu \times \mu))$ where the intersection is over all pmp actions $G \cc (X,\mu)$. 

We say that $\Gamma \cc (B,\nu)$ is {\em weakly mixing} if for any ergodic pmp action $\Gamma \cc (X,\mu)$, the product action $\Gamma \cc (B\times X,\nu\times\mu)$ is ergodic. If $\nu$ is also non-atomic then the possibilities for the stable ratio set and the corresponding stable type classification are:
\begin{displaymath}
\begin{array}{c|cc}
\textrm{ stable ratio set } & \textrm{ stable type }\\
\hline
\{1\} & II &\\
\{0,1,\infty\} & III_0&\\
\{0,\lambda^n, \infty: ~n\in \ZZ\} & III_\lambda & (0 < \lambda < 1)\\
 \lbrack 0,\infty \rbrack & III_1&
 \end{array}\end{displaymath}

\subsection{Examples of non-zero stable type}

Let us proceed to give two important examples of these notions. First, consider the case that $\Gamma$ is an irreducible lattice in a connected semisimple Lie group $G$ which has trivial center and no compact factors. Consider the action of  $\Gamma$ on $(G/P,\nu)$, where $P<G$ is a minimal parabolic subgroup and $\nu$ is a probability measure in the unique $G$-invariant measure class. As noted in \cite{BN2}, it is well-known that the action is weakly-mixing, and furthermore, it is shown in \cite{BN2} that the action of $\Gamma$ on $G/P$ is stable type $III_1$. Second, in \cite{B2} 
it is shown that for the action of a non-elementary word-hyperbolic group on its Gromov boundary, under suitable assumptions the Patterson-Sullivan measure gives rise to a weakly-mixing action which is of stable type $III_\lambda$, with $0 < \lambda \le 1$. A curious example occurs by considering the action of the free group $\FF_r$ of rank $r$ on its Gromov boundary with respect to the usual Patterson-Sullivan measure (see \S \ref{sec:free} for details). This action is type $III_\lambda$ and stable type $III_{\lambda^2}$ where $\lambda = (2r-1)^{-1}$. These are the only results on stable type of which we are aware.


By comparison, in \cite{INO08} it is proven that the Poisson boundary of a random walk on a Gromov hyperbolic group induced by a nondegenerate measure on $\Gamma$ of finite support is never of type $III_0$. In \cite{Su78, Su82}, Sullivan proved that the recurrent part of an action of a discrete conformal group on the sphere $\mathbb{S}^d$ relative to the Lebesgue measure is type $III_1$. Spatzier \cite{Sp87} showed that if $\Gamma$ is the fundamental group of a compact connected negatively curved manifold then the action of $\Gamma$ on the sphere at infinity of the universal cover is also of $III_1$. The types of harmonic measures on free groups were computed by Ramagge and Robertson \cite{RR97} and Okayasu \cite{Ok03}. 


\subsection{Maharam Extensions}\label{sec:maharam}

An interesting source of $\sigma$-finite measure-preserving amenable actions is obtained as follows. Let $\Gamma \cc (B,\nu)$ be a non-singular action on a probability space. Let 
$$R(g,b) =\log\left( \frac{ d\nu \circ g}{d\nu}(b)\right).$$
Define a measure $\theta$ on $\RR$ by $d\theta(t) = e^{t} dt$. Then $\Gamma$ acts on $B\times \RR$ by
$$g (b,t) = (gb, t- R(g,b)).$$
This action is called the {\em Maharam extension}. It preserves the $\sigma$-finite measure $\nu\times \theta$. If $\Gamma \cc (B,\nu)$ is amenable then this action is also amenable. 

If the Radon-Nikodym derivative takes values in a proper subgroup of $\RR$ then it is more appropriate to consider the {\em discrete Maharam extension}, which is defined as follows. For $\lambda \in (0,1)$, let 
$$R_\lambda(g,b) =\log_\lambda\left( \frac{ d\nu \circ g}{d\nu}(b)\right).$$
Suppose that $R_\lambda(g,b) \in \ZZ$ for every $g\in \Gamma$ and a.e. $b \in B$. Define a measure $\theta_\lambda$ on $\ZZ$ by $\theta_\lambda(\{n\}) = \lambda^{-n}$. Then $\Gamma$ acts on $B\times \ZZ$ by
$$g (b,t) = (gb, t+ R_\lambda(g,b)).$$
This action is called the {\em discrete Maharam extension}. It preserves the $\sigma$-finite measure $\nu\times \theta_\lambda$. If $\Gamma \cc (B,\nu)$ is amenable then this action is also amenable.

\section{Stable type and weakly-mixing relative to a compact group}\label{sec:typeIII}

In this section, we will demonstrate how to utilize the hypothesis that a non-singular action has a non-zero stable type in order to show that its Maharam extension gives rise to an ergodic pmp equivalence relation which is weakly-mixing relative to a compact (in fact, a circle) group. 

\subsection{
Statement of Theorem \ref{thm:typeIII}}

Suppose $\Gamma \cc (B,\nu)$ is an essentially free, weakly mixing action of stable type $III_1$. In this case, the Maharam extension $\Gamma \cc (B\times \RR, \nu\times \theta)$ is weakly mixing. This fact is implied by \cite[Corollary 4.3]{BN2} which is a straightforward application of \cite[Proposition 8.3, Theorem 8]{FM77}. Thus if $B_0 \subset B\times \RR$ is any Borel set with $\nu\times \theta(B_0)=1$ and $\alpha: B_0 \to \Gamma$ is defined by $\alpha(c,b)=\gamma$ where $\gamma b = c$ then $\alpha$ is weakly mixing and therefore, we can apply Theorem \ref{thm:general}.

Since our goal is to utilize amenable actions of $\Gamma$ to construct ergodic averages on $\Gamma$, it becomes necessary to consider non-singular actions with type $III_1$ and stable type $\tau\in (0,1)$. In that case, there arises naturally a non-singular action which is weakly-mixing relative to a compact group $K$, and we will be able to apply Theorem \ref{thm:general} again. We therefore formulate the following result, which will be proved next.

\begin{thm}\label{thm:typeIII}
Let $\Gamma \cc (B,\nu)$ be a nonsingular essentially free weakly mixing action on a standard probability space of type $III_1$ and stable type $III_\lambda$ for some $\lambda \in (0,1)$. Let $T=-\log(\lambda)$, and $\cR$ be the equivalence relation on $B \times [0,T]$ obtained by restricting the orbit-equivalence relation on the Maharam extension. Thus, $(b,t)\cR(b',t')$ if there exists $\gamma \in \Gamma$ such that $(\gamma b, t- R(\gamma,b)) = (b',t')$. Let $\alpha:\cR \to \Gamma$ be the cocycle $\alpha((b',t'), (b,t))=\gamma$ if $\gamma(b,t)=(b',t')$. Also, let $K =\RR/T\ZZ$ and $K \cc B\times [0,T]$ be the action $(r + T\ZZ)(b,t) = (b,t+r)$ (where $t+r$ is taken modulo $T$). Then $\alpha$ is weakly mixing relative to this $K$-action.
\end{thm}

We remark that the case where the type of the action is $\lambda\in (0,1)$ and the stable type is $\tau\in (0,1)$ involves the discrete Maharam extension and was considered in \cite{BN2}.  For completeness we state this result explicitly. 
\begin{thm}\label{previous}\cite[Corollary 5.4]{BN2}. 
Let $\Gamma \cc (B,\nu)$ be a nonsingular essentially free weakly mixing action on a standard probability space of type $III_\lambda$ and stable type $III_\tau$ for some $\lambda,\tau \in (0,1)$. Assume $R_\lambda(g,b)\in \ZZ$ for every $g\in G, b \in B$  and $\tau=\lambda^N$ for some $N\ge 1$. Let $\cR$ be the equivalence relation on $B \times \{0,1,\ldots, N-1\}$ obtained by restricting the orbit-equivalence relation on the discrete Maharam extension. Thus, $(b,t)\cR(b',t')$ if there exists $\gamma \in \Gamma$ such that $(\gamma b, t+ R_\lambda(\gamma,b)) = (b',t')$. Let $\alpha:\cR \to \Gamma$ be the cocycle $\alpha((b',t'), (b,t))=\gamma$ if $\gamma(b,t)=(b',t')$. Also, let $K =\ZZ/N\ZZ$ and $K \cc B\times \{0,1,\dots, N-1\}$ be the action $(r + T\ZZ)(b,t) = (b,t+r)$ (where $t+r$ is taken modulo $N$). Then $\alpha$ is weakly mixing relative to this $K$-action.
\end{thm}

We remark that the notation in \cite[Corollary 5.4]{BN2} is somewhat different : in that paper $\EE[f|\cI(\tilde{\cR_I})]$ denotes $\EE[f | {\cR_\alpha}]$ and $\theta_{\lambda,I}$ is the probability measure on $\{0,1,\ldots, N-1\}$ given by $\theta_{\lambda,I}(\{j\}) = \frac{\lambda^{-j}}{1+ \lambda^{-1}+\cdots \lambda^{-N+1}}$.

\subsection{
Proof of Theorem \ref{thm:typeIII}}

The proof of Theorem \ref{thm:typeIII} utilizes the construction of the Mackey range of the Radon-Nikodym cocycle associated with the measure $\nu\times \mu$ on $B\times X$, via the following result. 
\begin{prop}\label{prop:erg}
Suppose $\Gamma \cc (B,\nu)$ is an essentially free, ergodic action of type $III_1$.  Let $\Gamma \cc (X,\mu)$ be an ergodic pmp  action. Suppose that the product action $\Gamma \cc (B\times X,\nu\times\mu)$ is ergodic and type $III_\tau$ for some $\tau \in (0,1)$. Let $\Gamma$ act on $B \times X \times \RR$ by 
$$g(b,x,t) = \left(gb,gx, t -R(g,b) \right)$$
where 
$$R(g,b) = \log \left( \frac{d\nu\circ g}{d\nu}(b) \right).$$
Let $T_0=-\log(\tau), I \subset \RR$ be a  compact interval, $\widetilde{\cR}_I$ be the equivalence relation on $B \times X \times I$ obtained by restricting the orbit-equivalence relation. Then there is a measurable map $\phi:B \times X \times \RR \to \RR/T_0\ZZ$ defined almost everywhere, and a measurable family of probability measures $\{\eta_z:~ z \in \RR/T_0\ZZ\}$ defined almost everywhere, satisfying
\begin{enumerate}
\item $\phi(b,x,t+t') \equiv \phi(b,x,t) + t' \mod T_0$ for a.e. $(b,x,t)$ and every $t'$;
\item $\phi$ is essentially $\Gamma$-invariant;
\item almost every $\eta_z$ is an ergodic $\widetilde{\cR}_I$-invariant probability measure on $B\times X \times I$;
\item almost every  $\eta_z$ is supported on $\phi^{-1}(z)$;
\item $T_0^{-1} \int_0^{T_0} \eta_z ~dz = \nu \times \mu \times \theta_I$ where $\theta_I$ is the restriction of $\theta$ to $I$ normalized to have total mass $1$.
\end{enumerate}
\end{prop}


\begin{proof}

Given any probability space $(U,\eta)$, let ${\mathcal M} (U,\eta)$ denote the associated measure algebra, consisting of equivalence classes of measurable sets modulo null sets. Let $\sigma$ be a probability measure on $\RR$ that is equivalent to Lebesgue measure. Let ${\mathcal M}_\Gamma$ be the collection of all sets $A \in {\mathcal M}(B\times X \times \RR, \nu\times\mu\times \sigma)$ that are essentially $\Gamma$-invariant, modulo null sets. Note that $\RR$ acts on $B\times X \times \RR$ by $(t, (b,x,t')) \mapsto (b,x,t+t')$ and this action commutes with the $\Gamma$-action. By Mackey's point realization theorem (see e.g., \cite[Corollary B.6]{Zi84}), there is a standard Borel probability space $(Z,\kappa)$, a measurable almost everywhere defined  map $\phi: B\times X \times \RR \to Z$ and a nonsingular action $\{S_t\}_{t\in \RR}$ of $\RR$ on $(Z,\kappa)$ such that
\begin{enumerate}
\item ${\mathcal M}_\Gamma$ is the pullback under $\phi$ of the measure algebra  ${\mathcal M}(Z,\kappa)$;
\item $\phi( g (b,x,t) ) = \phi(b,x,t)$ for a.e. $(b,x,t)$ and every $g\in \Gamma$;
\item $\phi( b,x, t + t') = S_t\phi(b,x,t')$ for a.e. $(b,x,t')$ and every $t \in \RR$.
\end{enumerate}

Because $\Gamma \cc (B\times X, \nu \times \mu)$ is ergodic, \cite[Proposition 8.1]{FM77} implies that the flow $\{S_t\}_{t\in \RR}$ on $(Z,\kappa)$ is also ergodic. The fact that the action $\Gamma \cc (B \times X, \nu \times \mu)$ is type $III_\tau$ is equivalent to the statement that the flow $\{S_t\}_{t\in \RR}$ on $(Z,\kappa)$ is isomorphic to the canonical action of $\RR$ on $\RR/T_0\ZZ$ (see \cite[Theorem 8]{FM77}). So without loss of generality, we will assume $Z=\RR/T_0\ZZ$ and $\phi(b,x,t+t') = \phi(b,x,t)+t'$ for a.e. $(b,x,t)$ and every $t'$. The measure $\kappa$ on $Z=\RR/T_0\ZZ$ is quasi-invariant under this action.

Because $\Gamma \cc (B,\nu)$ is type $III_1$ (and ergodic), it follows that the action $\Gamma \cc (B\times \RR, \nu \times \theta)$ given by $g (b,t) = (gb, t- R(g,b))$ is ergodic \cite[Theorem 8]{FM77}. For almost every $(b,t)\in B\times \RR$, let $\mu_{(b,t)}$ be the pushforward of $\mu$ under the map $x \mapsto \phi(b,x,t)$. We claim that $\mu_{g(b,t)} = \mu_{(b,t)}~\forall g\in \Gamma$ and almost every $(b,t)$. Indeed, $\mu_{g(b,t)}$ is the pushforward of $\mu$ under the map 
$$x \mapsto \phi(gb, x, t- R(g,b)) = \phi( b, g^{-1}x, t- R(g,b) - R(g^{-1},gb)) = \phi(b, g^{-1}x, t)$$
where we have used the fact that $\phi$ is essentially $\Gamma$-invariant and $R(g,b) + R(g^{-1},gb)=0$ by the cocycle equation. Because $\mu$ is $\Gamma$-invariant, this implies the claim: $\mu_{g(b,t)} = \mu_{(b,t)}$ for a.e. $(b,t)$ and every $g\in \Gamma$. 

Because the action $\Gamma \cc (B\times \RR, \nu \times \theta)$ is ergodic, and the map $\mu_{(b,t)}\mapsto \mu_{g(b,t)}$ is essentially invariant, there is a measure $\rho$ on $\RR/T_0\ZZ$ such that $\mu_{(b,t)} = \rho$ for almost every $(b,t)$. We claim that $\rho$ is the Haar probability measure on $\RR/T_0\ZZ$. To see this, note that for almost every  $(b,t)$ and every $g\in \Gamma$, $\mu_{(b,t)} = \mu_{(gb,t)} = \rho$. However, $\mu_{(gb,t)}$ is the pushforward of $\mu$ under the map
$$x \mapsto \phi(gb,x,t) = \phi(b,g^{-1}x, t- R(g^{-1},gb)) = \phi(b,g^{-1}x, t) +R(g,b).$$
Because $\mu$ is $\Gamma$-invariant, this is the same as the pushforward of $\mu$ under the map $x \mapsto \phi(b,x,t) + R(g,b)$. In other words, it is the pushforward of $\mu_{(b,t)}$ under the map $t' \mapsto t' +R(g,b)$. Thus we have that $\rho$ is invariant under addition by $R(g,b)$ for a.e. $b\in B$ and $g\in \Gamma$. Because $\Gamma \cc (B,\nu)$ is type $III_1$, this implies that $\rho$ is $\RR$-invariant. So it is Haar probability measure as claimed.

Let $I \subset \RR$ be a compact interval. Because $\mu_{(b,t)}=\rho$ for a.e. $(b,t)$, it follows $\rho=\phi_*( \nu \times \mu \times \theta_I)$. Let $\{\eta_z:~ z\in \RR/T_0\ZZ\}$ be the decomposition of $\nu\times\mu\times\theta_I$ over $(\RR/T_0\ZZ,\rho)$. This means that 
\begin{enumerate}
\item almost  $\eta_z$ is a probability measure on $B\times X\times I$ with $\eta_z( \{ (b,x,t):~ \phi(b,x,t)=z\}) = 1$;
\item $\nu\times\mu\times\theta_I = \int \eta_z~d\rho(z)$.
\end{enumerate}
By definition of $\phi$, each $\eta_z$ is $\widetilde{\cR}_I$-invariant and ergodic. This implies the proposition.

\end{proof}


\begin{proof}[Proof of Theorem \ref{thm:typeIII}]
Recall that $\Gamma \cc (B,\nu)$ has stable type $III_\lambda$, $\lambda \in (0,1)$ and $T=-\log(\lambda)$. Let $\Gamma \cc (X,\mu)$ be an ergodic action and $\widetilde{\cR}_I$ is the corresponding equivalence relation on  $B\times X \times [0,T]$. So $(b,x,t) \widetilde{\cR}_I (b',x', t')$ if there exists $\gamma \in \Gamma$ such that $(b',x', t')=(\gamma b,\gamma x, t - R(\gamma,b))$.


Because $\Gamma \cc (B,\nu)$ is stable type $III_\lambda$, the ratio set of $\Gamma \cc (B\times X,\nu\times\mu)$ contains $\{\lambda^n:~n\in\ZZ\}$. Thus $\lambda = \tau^n$ for some $n\ge 1$, $n \in \ZZ$. Let $T_0=-\log(\tau)$. Because $\Gamma \cc (B,\nu)$ is weakly-mixing, $\Gamma \cc (B\times X,\nu \times \mu)$ is ergodic. So Proposition \ref{prop:erg} implies the existence of a map $\phi: B \times X \times \RR \to \RR/T_0\ZZ$ and probability measures $\{\eta_z:~z \in \RR/T_0\ZZ\}$ on $B\times X \times I$ satisfying conditions (1-5) of the proposition. So $\nu\times \mu \times \theta_I = T^{-1} \int_0^T \eta_z~dz$ is the ergodic decomposition of $\nu\times \mu\times \theta_I$ (w.r.t. $\widetilde{\cR}_I$). Thus, for any $f\in L^1(B\times X \times I)$,
$$\EE_{\nu\times\mu\times\theta_I}[ f | {\widetilde{\cR}_I} ] (b,x,t) =  \int f ~d\eta_{\phi(b,x,t)}.$$
Because $\phi$ is equivariant with respect to the $\RR$-action, it follows that for a.e. $(b,x)$,
$$\frac{1}{T_0} \int_0^{T_0} \eta_{\phi(b,x,t)} ~dt = \int \eta_z ~d\rho(z) = \nu\times\mu\times\theta_I$$
where $\rho$ is the Haar probability measure on $\RR/T_0\ZZ$. Because $T=n T_0$,
\begin{eqnarray*}
T^{-1} \int_0^T \EE_{\nu\times\mu\times\theta_I}[f| {\widetilde{\cR}_I}|(b,x,t)~dt &=& T^{-1}\int_0^{T} \int f ~d\eta_{\phi(b,x,t)} dt\\
&=& T_0^{-1}\int_0^{T_0} \int f ~d\eta_{\phi(b,x,t)} dt= \int f ~d\nu\times\mu\times\theta_I
\end{eqnarray*}
as required.

\end{proof}


\section{Cocycles and ergodic theorems : Some examples}\label{sec:example}

\subsection{Random walks, the associated cocycle, and ergodic theorems for convolutions} 
The construction in the present section is of a weakly-mixing cocycle on an ergodic amenable pmp equivalence relation {\it for an arbitrary countable group}. 
\begin{thm}\label{thm:existence}
For any countably infinite group $\Gamma$, there exists a weakly mixing cocycle $\alpha:\cR \to \Gamma$, where $\cR\subset B\times B$ is an amenable discrete ergodic pmp equivalence relation.
\end{thm}

\begin{proof}


Choose a probability measure $\kappa$ on $\Gamma$ whose support generates $\Gamma$. Consider the product space $\Gamma^\ZZ$ with the product topology. Let $\kappa^\ZZ$ be the product measure on $\Gamma^\ZZ$. Let $T: \Gamma^\ZZ \to \Gamma^\ZZ$ be the usual shift action $Tx(n):=x(n+1)$. Let $\cR$ be the equivalence relation on $\Gamma^\ZZ$ determined by the orbits of $\inn{T}\cong \ZZ$, namely $\cR=\{ (x,T^nx):~ x\in \Gamma^\ZZ, n \in \ZZ\}=\cO_\ZZ(S^\ZZ)$. Dye's Theorem \cite{Dy59, Dy63} implies $\cR$ is the hyperfinite $II_1$ equivalence relation (modulo a measure zero set), so that it is amenable \cite{CFW}. 

Define $\alpha:\cR\to \Gamma$ by 
\begin{displaymath}
\alpha(x,T^nx)= \left\{ \begin{array}{ll}
x(1)x(2)\cdots x(n) & \textrm{ if } n > 0 \\
e & \textrm{ if } n = 0\\
x(0)^{-1}x(-1)^{-1}\cdots x(n+1)^{-1} & \textrm{ if } n <0
\end{array}\right.
\end{displaymath}
$\alpha$ is a measurable cocycle well-defined on an invariant conull  measurable set. 

To show $\alpha$ is weakly mixing, let $\Gamma \cc (X,\mu)$ be a pmp action. Define $\tilde{T}:\Gamma^\ZZ \times X \to \Gamma^\ZZ \times X$ by $\tilde{T}(x,y) = (Tx, x(0)^{-1}y)$. Note that the relation $\cR(X) $ is the orbit-equivalence relation of $\tilde{T}$ on $\Gamma^\ZZ\times X$. Therefore, $\cR(X)$ is ergodic if and only if $\tilde{T}$ is ergodic. Kakutani's random ergodic theorem \cite{Ka55} implies $\tilde{T}$ is indeed ergodic.
\end{proof}

What then is the  pointwise ergodic theorem arising from our construction above ?  To answer this, recall that if $\kappa_1,\kappa_2$ are two probability measures on $\Gamma$ then their convolution $\kappa_1*\kappa_2$ is a probability measure on $\Gamma$ defined by
$$\kappa_1\ast \kappa_2(\{g\}) = \sum_{h \in \Gamma} \kappa_1(\{gh^{-1}\})\kappa_2(\{h\}).$$
Let $\kappa^{*n}$ denote the $n$-fold convolution power of $\kappa$.

\begin{cor}
Let $\kappa$ be a probability measure on $\Gamma$ whose support generates $\Gamma$. Let $\rho_n = \frac{1}{n} \sum_{k=1}^n \kappa^{\ast k}$. Then $\{\rho_n\}_{n=1}^\infty$ is a pointwise ergodic sequence in $L\log L$.
\end{cor}


\begin{proof}

Define $\alpha:\cR \to \Gamma$ as in the proof of Theorem \ref{thm:existence}. Let $\Omega=\{\omega_n\}_{n=1}^\infty$ be the sequence of leafwise probability measures on $\cR$ defined by $\omega_n(x,y) = \frac{1}{n}$ if $y= T^i x $ for some $1\le i \le n$. Because $T$ is ergodic, and these measures are the ergodic averages on $T$-orbits, this sequence satisfies the weak (1,1)-type maximal inequality and is pointwise ergodic in $L^1$, by Birkhoff's and Wiener's Theorem.

By Theorem \ref{thm:general}, if we set 
$$\zeta_n(\gamma):= \int_{\Gamma^\ZZ} \sum_{c:~\alpha(c,b)=\gamma} \omega_n(c,b) ~d\kappa^\ZZ(b) $$
then $\{\zeta_n\}_{n=1}^\infty$ is a pointwise ergodic family in $L\log(L)$. To see this, set $K$ equal to the trivial group and $\psi \equiv 1$. 

The corollary now follows from a short calculation:
\begin{eqnarray*}
 \int \sum_{c:~\alpha(c,b)=\gamma} \omega_n(c,b) ~d\kappa^\ZZ(b) &=& \frac{1}{n} \sum_{i=1}^n \kappa^\ZZ(\{x \in \Gamma^\ZZ:~\alpha(x,T^i x) = \gamma\})\\
 &=&\frac{1}{n} \sum_{k=1}^n \kappa^{\ast k}(\{\gamma\}) = \rho_n(\{\gamma\}).
 \end{eqnarray*}
 \end{proof}
Thus this argument gives a different proof of the pointwise ergodic theorem for the uniform averages of convolution powers, which follows (for example) from  the Chacon-Ornstein theorem or the Hopf-Dunford Schwartz theorem on the pointwise convergence of uniform averages of powers of a Markov operator.

\subsection{The free group}
We proceed to demonstrate our approach by obtaining pointwise ergodic theorems from the action of the free group on its boundary and its double-boundary.


\subsubsection{The boundary action}\label{sec:free}
Let $\FF=\FF_r=\langle a_1,\dots ,a_r \rangle$ be the free group of rank $r\ge 2$, and $S=\{a_i^{\pm 1}:~1\le i \le r\}$ a set of free generators. The {\em reduced form} of an element $g\in \FF$ is the unique expression  $g=s_1\cdots s_n$ with $s_i \in S$ and $s_{i+1}\ne s_i^{-1}$ for all $i$.  Define $|g|:=n$, the length of the reduced form of $g$ and the sphere $S_n=\set{w\,;\, \abs{w}=n}$. 

We identify the boundary of $\FF$ with the set of all sequences $\xi=(\xi_1,\xi_2,\ldots) \in S^\NN$ such that $\xi_{i+1} \ne \xi_i^{-1}$ for all $i\ge 1$, and denote it by $\partial \FF$. Let $\nu$ be the probability measure on $\partial \FF$ determined  as follows. For every finite sequence $t_1,\ldots, t_n$ with $t_{i+1} \ne t_i^{-1}$ for $1\le i <n$, 
$$\nu\Big(\big\{ (\xi_1,\xi_2,\ldots) \in \partial \FF :~ \xi_i=t_i ~\forall 1\le i \le n\big\}\Big) := |S_n|^{-1}=(2r-1)^{-n+1}(2r)^{-1}.$$
By Carath\'eodory's Extension Theorem, this uniquely defines $\nu$ on the Borel sigma-algebra of $\partial \FF$.

The action of $\FF$ on $\partial \FF$ is given by 
$$(t_1\cdots t_n)\xi := (t_1,\ldots,t_{n-k},\xi_{k+1},\xi_{k+2}, \ldots)$$ where $t_1,\ldots, t_n \in S$,  $g=t_1\cdots t_n$ is in reduced form and $k$ is the largest number $\le n$ such that $\xi_i^{-1} = t_{n+1-i}$ for all $i\le k$.  In this case the Radon-Nikodym derivative satisfies
$$\frac{d\nu \circ g}{d\nu}(\xi) = (2r-1)^{2k-n}.$$
Thus if $\lambda=(2r-1)^{-1}$ then $R_\lambda(g,\xi) = 2k-n \in \Z$. 






The type of the action $\FF\cc (\partial \FF, \nu)$ is easily seen to $III_\lambda$ where $\lambda=(2r-1)^{-1}$. It is shown implicitly in \cite[Theorem 4.1]{BN1} that the stable type of $\FF\cc (\partial \FF, \nu)$ is $III_{\lambda^2}$. In fact, if $\FF^2$ denotes the index 2 subgroup of $\FF$ consisting of all elements $g$ with $|g| \in 2\Z$ then $\FF^2 \cc (\partial \FF,\nu)$ is type $III_{\lambda^2}$ and stable type $III_{\lambda^2}$. It is also weakly mixing. Indeed, $(\partial \FF,\nu)$ is naturally identified with $P(\FF,\mu)$, the Poisson boundary of the random walk generated by the measure $\mu$ which is uniformly distributed on $S$.  By \cite{AL}, the action of any countable group $\Gamma$ on the Poisson boundary $P(\Gamma,\kappa)$ is weakly mixing whenever the measure $\kappa$ is adapted. This shows that $\FF \cc (\partial \FF,\nu)$ is weakly mixing. Moreover, if $S^2 = \{st:~s, t \in S\}$ then $(\partial \FF,\nu)$ is naturally identified with the Poisson boundary $P(\FF^2, \mu_2)$ where $\mu_2$ is the uniform probability measure on $S^2$. So the action $\FF^2 \cc (\partial \FF,\nu)$ is also weakly mixing.


Let $\FF$ act on $\partial \FF \times \ZZ$ by $g(b,t) = (gb, t+ R_\lambda(g,b))$. This is the Maharam extension. Although it is possible to use Theorem \ref{previous} to obtain an equivalence relation on $\partial \FF \times \{0,1\}$ and a cocycle, it is more fruitful to consider a slightly different construction. Let $\cR$ be the orbit-equivalence relation restricted to $\partial \FF \times \{0\}$, which we may, for convenience, identify with $\partial \FF$. In other words, $b \cR b'$ if and only if there is an element $g \in \FF$ such that $gb=b'$ and $\frac{d\nu \circ g}{d\nu}(b)=1$. Observe that this is the same as the (synchronous) tail-equivalence relation on $\FF$. In other words, two elements $\xi=(\xi_1,\xi_2,\ldots)$, $\eta=(\eta_1,\eta_2,\ldots) \in \partial \FF$ are $\cR$-equivalent if and only if there is an $m$ such that $\xi_n=\eta_n$ for all $n \ge m$. Also note that if $b\cR (gb)$ then necessarily $g \in \FF^2$. So $\cR$ can also be regarded as the orbit-equivalence relation for the $\FF^2$-action on $\partial \FF \times \ZZ$ restricted to $\partial \FF \times \{0\}$. 

Let $\alpha:\cR \to \FF^2$ be the cocycle $\alpha(gb,b)=g$ for $g\in \FF^2, b \in \partial \FF$. This is well-defined almost everywhere because the action of $\FF^2$ is essentially free. Because $\FF^2 \cc (\partial \FF,\nu)$ has type $III_{\lambda^2}$ and stable type $III_{\lambda^2}$, this cocycle is weakly mixing for $\FF^2$. In other words, if $\FF^2 \cc (X,\mu)$ is any ergodic pmp action then the equivalence relation $\cR_\alpha$ defined on $\partial \FF \times X$ by $g(b,x)\cR_\alpha (b,x)$ if $\frac{d\nu \circ g}{d\nu}(b)=1$ is ergodic. This is the import of \cite[Theorem 4.1]{BN1}.

Now let $\omega_n: \cR \to [0,1]$ be the leafwise probability measure given by
\begin{displaymath}
\omega_n( gb,b ) =  \left\{ \begin{array}{cc}
(2r-2)^{-1}(2r-1)^{-n+1} & \textrm{ if } |g|=2n \\
0 & \textrm{ otherwise} \end{array}\right.
\end{displaymath}
In other words, $\omega_n(\cdot, b)$ is uniformly distributed over the set of all elements of the form $gb$ with $|g|=2n$ and $\frac{d\nu \circ g}{d\nu}(b)=1$. In \cite[Corollary 5.2 \& Proposition 5.3]{BN1} it is shown that $\Omega=\{\omega_n\}_{n=1}^\infty$ is pointwise ergodic in $L^1$ and satisfies the weak (1,1)-type maximal inequality. 

It follows from Theorem \ref{thm:general} that if
$$\zeta_n(\gamma):= \int_{\partial \FF} \sum_{c:~\alpha(c,b)=\gamma} \omega_n(c,b)  ~d\nu(b)$$
then $\{\zeta_n\}_{n=1}^\infty$ is pointwise ergodic in $L\log(L)$ for $\FF_2$-actions. A short calculation reveal that $\zeta_n$ is uniformly distributed on the sphere of radius $2n$. This yields:
\begin{thm}\label{thm:sphere}
Let $\FF \cc (X,\mu)$ be a pmp action. Then for any $f \in L\log L(X,\mu)$, 
$$|S_{2n}|^{-1} \sum_{|g|=2n} f\circ g$$
converges pointwise a.e as $n\to\infty$ to $\EE[f| \FF^2]$ the conditional expectation of $f$ on the sigma-algebra of $\FF^2$-invariant measurable subsets.
\end{thm}
This theorem was proven earlier in \cite{NS} and \cite{Bu} by different techniques.

\subsubsection{The double boundary action}

Next we analyze the action of $\FF$ on its double boundary $\partial^2 \FF := \partial \FF \times \partial \FF$. To begin, let $\bnu$ be the probability measure on the double boundary $\partial^2\FF$ defined by
$$d\bnu(b,c) = (2r)(2r-1)^{2d-1}d\nu(b)d\nu(c)$$
where $d \ge 0$ is the largest integer such that $b_i=c_i$ for all $i\le d$. A short calculation shows that $\bnu$ is $\FF$-invariant. Because the action $\FF \cc (\partial \FF,\nu)$ is a strong $\FF$-boundary in the sense of \cite{Ka}, it follows that $\FF \cc (\partial^2 \FF, \bnu)$ is weakly mixing. 

We need to restrict the orbit-equivalence relation on $\partial^2 \FF$ to a fundamental domain. For convenience we choose the domain $D$ equal 
to the set of all $(b,c) \in \partial^2 \FF$ such that the geodesic from $b$ to $c$ passes through the identity element. In other words, $b_1 \ne c_1$. Observe that $\bnu(D)=\left(\frac{2r}{2r-1}\right)\nu \times \nu(D) = 1.$  

It is clarifying to view $D,\bnu$ and the orbit-equivalence relation from a slightly different point of view. To do this, let $\Phi:D \to S^\Z$ be the map 
\begin{displaymath}
\Phi(b,c)_i = \left\{\begin{array}{cc}
b_{-i}^{-1} & i <0 \\
c_{i+1} & i \ge 0
\end{array}\right.\end{displaymath}
Note that this is an embedding so we can identify $D$ with its image in $S^\Z$. Moreover $\Phi(D)$ is the set of all sequences $\xi \in S^\Z$ such that $\xi_i \ne \xi_{i+1}^{-1}$ for all $i$. Let $D'$ denote $\Phi(D)$.

Let $\nu' = \Phi_*\bnu$. Observe that $\nu'$ is a Markov measure. Indeed let $m<n$ be integers and for each $i$ with $m \le i \le n$, let $s_i \in S$ so that $s_i \ne s_{i+1}^{-1}$ for $m \le i < n$. Let $C(s_m,\ldots, s_n)$ be the cylinder set 
$$C(s_m,\ldots, s_n)=\{\xi \in D':~ \xi_i = s_i ~\forall m \le i \le n\}.$$
Then $\nu'(C(s_m,\ldots, s_n)) = (2r)^{-1}(2r-1)^{m-n}$. This formula completely determines $\nu'$ by Carath\'eodory's Extension Theorem.

If $g \in \FF$, $(b,c) \in D$ and $g(b,c) = (gb,gc) \in D$ then $(gb)_1 \ne (gc)_1$ implies that we cannot have partial cancellation of $g$ in both products $gb$ and $gc$. In other words, we must either have that $g=(b_1\cdots b_n)^{-1}$  or $g = (c_1 \cdots c_n)^{-1}$ for some $n$. Let us suppose that $g=(b_1\cdots b_n)^{-1}$. Then $gb = (b_{n+1},b_{n+2},\ldots)$, $gc = (b_n^{-1}, b_{n-1}^{-1},\ldots, b_1^{-1},c_1,c_2,\ldots)$ (this uses that $b_1\ne c_1$ since $(b,c)\in D$) and
$\Phi(gb,gc)_i = \Phi(b,c)_{i-n}$ for every $i$. Similarly, if $g= (c_1 \cdots c_n)^{-1}$ then $\Phi(gb,gc)_i = \Phi(b,c)_{i+n}$ for every $i$.

Let $\cR$ be the equivalence relation on $D$ obtained by restricting the orbit-equivalence relation on $\partial^2 \FF$. The previous paragraph implies that if $\cR'$ is the relation on $D'=\Phi(D)$ obtained by pushing forward $\cR$ under $\Phi$ then $\cR'$ is the same as the equivalence relation determined by the shift. To be precise, define $T:D' \to D'$ by $(T\xi)_i = \xi_{i-1}$ for all $i$. Then  two sequences $\xi, \eta \in D'$ are $\cR'$-equivalent if and only if there is an $n\in \ZZ$ such that $\xi = T^n\eta$. Note that $\nu'$ is $T$-invariant and therefore $\cR'$-invariant.

Let $\alpha: \cR \to \FF$ be the canonical cocycle given by $\alpha(g(b,c), (b,c)) = g$ for every $g \in \FF, (b,c) \in D$.  We can push forward $\alpha$ under $\Phi$ to obtain a cocycle $\alpha': \cR' \to \FF$ given by 
\begin{displaymath}
\alpha'( T^n \xi, \xi) = \left\{ \begin{array}{cc}
(\xi_0 \cdots \xi_{n-1})^{-1} &  \textrm{ if } n\ge 1 \\
e & \textrm{ if } n=0\\
\xi_n \xi_{n+1}\cdots \xi_{-1} & \textrm{ if } n\le -1
\end{array}\right.\end{displaymath}

Because $\FF \cc (\partial^2 \FF,\bnu)$ is weakly mixing, it follows that $\alpha$ and therefore $\alpha'$ are weakly mixing cocycles. Let $\omega_n$ be the leafwise measure on $\cR'$ defined by $\omega_n(T^i \xi, \xi) = \frac{1}{n+1}$ if $0\le i \le n$ and $\omega_n = 0$ otherwise. By Birkhoff's ergodic theorem, $\Omega= \{\omega_n\}_{n=1}^\infty$ is pointwise ergodic in $L^1$ and satisfies the weak $(1,1)$-type maximal inequality. It now follows from Theorem \ref{thm:general} that if 
$$\zeta_n(\gamma):= \int_{D'} \sum_{\eta:~\alpha'(\eta,\xi)=\gamma} \omega_n(\eta,\xi) ~d\nu'(\xi) $$
then $\{\zeta_n\}_{n=1}^\infty$ is pointwise ergodic in $L\log(L)$. A short calculation shows that 
$$\zeta_n  = \frac{1}{n+1} \sum_{i=0}^n \sigma_n$$
where $\sigma_n$ is the uniform probability measure on the sphere of radius $n$ in $\FF$. So Theorem \ref{thm:general} implies:
\begin{thm}
Let $\FF \cc (X,\mu)$ be a pmp action. Then for any $f \in L\log L(X,\mu)$, 
$$ \frac{1}{n+1} \sum_{i=0}^n \sigma_n(f)$$
converges pointwise a.e as $n\to\infty$ to $\EE[f| \FF]$ the conditional expectation of $f$ on the sigma-algebra of $\FF$-invariant measurable subsets.
\end{thm}
This theorem was proven earlier in \cite{NS}, \cite{Gr} and \cite{Bu} by different techniques. Of course, it also follows directly from Theorem \ref{thm:sphere}.


\subsection{Hyperbolic groups}
We now turn to briefly describe the application of our approach to establishing ergodic theorems for hyperbolic groups in a geometric setting. Full details will appear in \cite{BN4}, and for simplicity we will mention here only a special case of the results established there.


\begin{thm}\label{thm:hyp}
Suppose $\Gamma$ acts properly discontinuously by isometries on a  $CAT(-1)$ space $(X,d_X)$. Suppose there is an $x \in X$ with trivial stabilizer. Define a metric $d$ on $\Gamma$ by $d(g,g'):=d_X(gx,g'x)$. Then there exists a family $\{\kappa_r\}_{r>0}$ of probability measures on $\Gamma$ such that
\begin{enumerate}
\item there is a constant $a>0$ so that each $\kappa_r$ is supported on the annulus $\{g\in \Gamma:~ d(e,g) \in [r-a,r+a]\}$,
\item $\{\kappa_r\}_{r>0}$ is a pointwise ergodic family in $L^p$ for every $p>1$ and in $L\log L$. 
\end{enumerate}
\end{thm}

\subsubsection{A brief outline : the equivalence relation, the cocycle, and the weights}
Let $\partial\Gamma$ denote the Gromov boundary of $(\Gamma,d)$. Via the Patterson-Sullivan construction, there is a quasi-conformal probability measure $\nu$ on $\partial\Gamma$. So there are constants $C,\fh>0$ such that
$$C^{-1}\exp(-\fh h_\xi(g^{-1})) \le \frac{d\nu \circ g}{d\nu}(\xi) \le C \exp(-\fh h_\xi(g^{-1}))$$
for every $g\in \Gamma$ and a.e. $\xi \in \partial \Gamma$ where $h_\xi$ is the horofunction determined by $\xi$. To be precise,
$$h_\xi(g)=\lim_{n\to\infty} d(\xi_n,g)-d(\xi_n,e)$$
where $\{\xi_n\}_{n=1}^\infty$ is any sequence in $\Gamma$ which converges to $\xi$. Properties of CAT(-1) spaces ensure that this limit exists \cite{BH99}.

The {\em type} of the action $\Gamma \cc (\partial \Gamma,\nu)$ encodes the essential range of the Radon-Nikodym derivative, and  \cite{B2} it is shown that this type is $III_\lambda$ for some $\lambda \in (0,1]$. If $\lambda \in (0,1)$,  set
$$R_\lambda(g,\xi)= - \log_\lambda \left( \frac{d \nu \circ g}{d\nu}(\xi) \right)\,,$$
and  set $R_1(g,\xi) = +\log \left( \frac{d \nu \circ g}{d\nu}(\xi) \right)$. Using standard results, it can be shown that if $\lambda \in (0,1)$ then we can choose $\nu$ so that $R_\lambda(g,\xi) \in \ZZ$ for every $g$ and a.e. $\xi$. 

In order to handle each case uniformly, set $L=\RR$ if $\lambda=1$ and $L=\ZZ$ if $\lambda \in (0,1)$. Then let $\Gamma$ act on $\partial\Gamma \times L$ by
$$g(\xi,t) = (g \xi, t-R_\lambda(g,\xi)).$$
This action preserves the measure $\nu \times \theta_\lambda$ where $\theta_1$ is the measure on $\RR$ satisfying $d\theta_1(t) = \exp(\fh t) dt$ and, for $\lambda \in (0,1)$, $\theta_\lambda$ is the measure on $\ZZ$ satisfying $\theta_\lambda(\{n\}) = \lambda^{-n}$. 

Given $a,b \in L$, let $[a,b]_L \subset L$ denote the interval $\{a,a+1,\ldots, b\}$ if $L=\ZZ$ and $[a,b] \subset \RR$ if $L=\RR$. Similar considerations apply to open intervals and half-open intervals.

For any real numbers $r, T>0$, and $(\xi,t) \in \partial\Gamma\times [0,T)_L$, let 
$$\Gamma_r(\xi,t) = \{g \in \Gamma:~ d(g,e) - h_\xi(g)-t \le r, ~g^{-1}(\xi,t) \in \partial \Gamma \times [0,T)_L \}$$
and
$$\cB^{}_r(\xi,t) := \{ g^{-1} (\xi,t):~ g\in \Gamma_r(\xi,t) \}.$$
$\Gamma_r(\xi,t)$ is approximately equal to the intersection of the ball of radius $r$ centered at the identity with the horoshell $\{g \in \Gamma:~ -t\le h_\xi(g) \le T-t\}$.  Of course, $\Gamma_r$ and $\cB_r$ depend on $T$, but we leave this dependence implicit. 

The main steps in the proof of Theorem \ref{thm:hyp} are as follows 
\begin{enumerate}
\item 
The first main technical result is that if $T$ are sufficiently large then the subset family $\cB$ is {\em regular}: there exists a constant $C>0$ such that for every $r>0$ and a.e. $(\xi,t) \in \partial\Gamma \times [0,T)_L$,
$$| \cup_{s\le r} \cB^{-1}_s \cB_r (\xi,t)| \le C |\cB_r(\xi,t)|.$$
The weak (1,1) type maximal inequality for the family of uniform averages on $\cB_r(\xi,t)$ is then established using  the general results in \cite{BN2}. 

\item 
Next we let $\cS^{}_a=\{\cS^{}_{r,a}\}_{r>0}$ be the family of subset functions on $\partial\Gamma\times [0,T)_L$ defined by 
$$\cS^{}_{r,a}(\xi,t) := \cB_r(\xi,t) \setminus \cB_{r-a}(\xi,t)$$
and observe that $\cS_a$ is also regular if $a,T>0$ are sufficiently large. Therefore the family of uniform averages on $\cS_{r,a}(\xi,t)$ satisfies a weak (1,1) type maximal inequality.

The second main technical result is that $\cS_a$ is {\em asymptotically invariant}. More precisely we let $E$ denote the equivalence relation on $\partial\Gamma \times [0,T)_L$ given by $(\xi,t)E(\xi',t')$ if there exists $g \in \Gamma$ such that $g (\xi,t)=(\xi',t')$. Recall that $[E]$ denotes the full group of $E$ (that is, the group of all Borel isomorphisms on $\partial \Gamma \times [0,T)_L$ with graph contained in $E$), and $\cS_{r,a}$ being asymptotically invariant means that there exists a countable set $\Psi \subset [E]$ which generates the relation $E$ such that 
$$\lim_{r \to \infty} \frac{ |\cS_{r,a}(\xi,t) \vartriangle \psi(\cS_{r,a}(\xi,t))|}{|\cS_{r,a}(\xi,t)|} = 0$$
for a.e. $(\xi,t)$ and for every $\psi \in \Psi$.


It now follows from the general results of \cite{BN2} that the uniform averages over $\cS_{r,a}$ form a pointwise ergodic family $\Omega=\{\omega_r\}_{r>0}$ in $L^1$ for the equivalence relation $E$. 

\item We let $\alpha:E \to \Gamma$ be the cocycle given by the action of $\Gamma$. First we show that $\Gamma \cc (\partial\Gamma,\nu)$ is weakly mixing (so $\Gamma \cc (X\times \partial\Gamma, \mu \times \nu)$ is ergodic). This uses the fact that Poisson boundary actions are weakly mixing \cite{AL, Ka} and that the action on $(\partial\Gamma,\nu)$   is equivalent to a Poisson boundary action \cite{CM07}. From \cite{B2}, it follows that $\Gamma \cc (\partial\Gamma,\nu)$ has type $III_\rho$ and stable type $III_\tau$ for some $\rho, \tau \in (0,1]$. Therefore $\alpha$ is weakly mixing relative to a compact group action. So we can invoke Theorem \ref{thm:general} and Theorem \ref{thm:typeIII} of the present paper  (as well as Theorem \ref{previous} based on \cite{BN2}) and thereby conclude the proof.
\end{enumerate}

\end{document}